\documentclass[10pt,a4paper]{article}

\usepackage[margin=.825in]{geometry}
\usepackage{fourier}
\usepackage{enumitem}
\usepackage{bbm}
\usepackage{amsmath,amsthm,amssymb}

\usepackage{caption}
\newtheorem{theorem}{Theorem}
\newtheorem{corollary}{Corollary}
\newtheorem{lemma}{Lemma}
\newtheorem{assumption}{Assumption}
\newtheorem{proposition}{Proposition}

\newcommand{\PP}{\mathbb{P}}
\newcommand{\F}{\mathcal{F}}
\newcommand{\Rd}{\mathbb{R}^{d}}
\newcommand{\Rn}{\mathbb{R}^{n}}
\newcommand{\Rm}{\mathbb{R}^{m}}
\newcommand{\R}{\mathbb{R}}
\newcommand{\Rnd}{\mathbb{R}^{n\times d}}
\newcommand{\Rnm}{\mathbb{R}^{n\times m}}
\newcommand{\E}{\mathbb{E}}
\newcommand{\LL}{\mathcal{L}}
\newcommand{\U}{\mathcal{U}}
\newcommand{\B}{\mathcal{B}}
\newcommand{\tr}{\mbox{tr}}
\newcommand{\uu}{\widehat{u}}
\newcommand{\X}{\widehat{X}}
\newcommand{\x}{\widehat{x}}
\newcommand{\var}{\mbox{Var}}

\date{}

\begin{document}

\title{Nonlinear stochastic receding horizon control: stability, robustness and Monte Carlo methods for control approximation\thanks{This work was supported by AFOSR/AOARD via AOARD-144042.}}

\author{Francesco Bertoli\thanks{F. Bertoli is with the Australian National University (ANU).} \and Adrian N. Bishop\thanks{A.N. Bishop is with the University of Technology Sydney (UTS) and CSIRO. He is also an adjunct Fellow at the Australian National University (ANU). He is supported by the Australian Research Council (ARC) via a Discovery Early Career Researcher Award (DE-120102873).}}

\maketitle

\begin{abstract}
This work considers the stability of nonlinear stochastic receding horizon control when the optimal controller is only computed approximately. A number of general classes of controller approximation error are analysed including deterministic and probabilistic errors and even controller sample and hold errors. In each case, it is shown that the controller approximation errors do not accumulate (even over an infinite time frame) and the process converges exponentially fast to a small neighbourhood of the origin. In addition to this analysis, an approximation method for receding horizon optimal control is proposed based on Monte Carlo simulation. This method is derived via the Feynman-Kac formula which gives a stochastic interpretation for the solution of a Hamilton-Jacobi-Bellman equation associated with the true optimal controller. It is shown, and it is a prime motivation for this study, that this particular controller approximation method practically stabilises the underlying nonlinear process.
\end{abstract}

\section{Introduction}

Receding horizon optimal control (RHC) is a strategy for controlling a dynamical system over an (possibly) infinite horizon where the control input at any instant is derived by solving a finite horizon optimal control problem over a fixed length horizon from that instant forwards. An introduction to RHC can be found in \cite{mayne2000constrained,kwon2006receding}. RHC is a natural extension of finite-horizon optimal control and a natural simplification of infinite-horizon optimal control. The term model predictive control is often used interchangeably with RHC.

The contribution of this work is:

\begin{description}[labelindent=0.5\parindent, leftmargin=2.5\parindent]

  \item[\ 1] We study the stability of continuous-time receding horizon control of nonlinear stochastic systems when the optimal controller computation is only approximate. In particular, we consider a number of classes of controller approximation error including deterministic and probabilistic errors. We also consider controller sample and hold errors, that arise due to real-time computing limitations etc.

  \item[2a] We outline a (Monte Carlo) simulation algorithm for approximating the optimal receding horizon control for nonlinear stochastic continuous-time systems. 

  \item[2b] We connect the controller approximation technique to the stability analysis detailed in this work and show that this Monte Carlo simulation method for controller approximation stabilises the process (in a sense to be made precise). In particular, the approximation errors do not accumulate nor destabilise the system.
\end{description}

The analysis of RHC for nonlinear (deterministic) systems started, largely, with the analysis of Mayne et. al. \cite{mayne1990receding,michalska1993robust}. Even in this early work, stability was considered for RHC in the presence of a number of controller approximation errors. Broad work on this topic in the nonlinear realm is covered in \cite{parisini1995receding,de1998stabilizing,jadbabaie2001unconstrained}. Much work in this area has focused on the incorporation of (deterministic) model uncertainty \cite{michalska1993robust,de1996robustness,magni2001receding,magni2003robust} and/or controller and state constraints \cite{michalska1993robust}. This latter focus concerning constraints is largely beyond the scope of this study, although we comment on possible extensions in our concluding remarks. 

In the stochastic realm, the foundations of optimal control of nonlinear (continuous-time) systems are studied in, e.g., \cite{fleming2006controlled,krylov2008controlled,yong1999stochastic,kushner2013numerical,touzi2012optimal}.
Computational methods for general nonlinear stochastic RHC are given in, e.g., \cite{parisini1998neural,kushner2013numerical,kappen2005path,mceneaney2006max,li2007iterative,kantas2009sequential,stahl2011pf,de2011particle,bierkens2014explicit}. In continuous-time cases, optimal nonlinear RHC has been shown \cite{wei2014stability} to be stabilising under the assumption that an optimal controller is applied (exactly). In this work, we extend \cite{wei2014stability} by showing that stability is retained even in the presence of controller approximation errors. We consider a number of general classes of controller approximation error, including deterministic and probabilistic errors and also controller sample and hold errors. To the best of our knowledge, there has been no investigation on the stochastic stability of (nonlinear stochastic) RHC in the presence of controller approximation errors. The challenge in this case is ensuring that solutions of the controlled diffusion are bounded `almost surely' within some neighbourhood of the origin; we employ a classical stochastic stability analysis (viz. \cite{khasminskii2011stochastic}) with a novel application of the optional sampling theorem \cite{doob1953stochastic}. Dealing with sample and hold-type errors is also challenging as the standard Euler-Maruyama time-discretisation of a stochastic differential equation can be unstable \cite{higham2003exponential}.

In addition to the stability analysis detailed in the preceding, we also outline an approximation method for computing the optimal RHC for nonlinear stochastic continuous-time systems. This method is based on Monte Carlo integral approximation and originates in the work of Kappen \cite{kappen2005linear,kappen2005path} where such techniques were applied in finite-horizon optimal control for nonlinear stochastic systems. The broad idea is that a solution to the Hamilton-Jacobi-Bellman partial differential equation associated with a typical optimal control problem \cite{fleming2006controlled} can be formulated in terms of an expectation over a stochastic trajectory defined by an uncontrolled stochastic differential equation (SDE). Indeed, this relationship between partial differential equations and so-called path-integrals is just a consequence of Feynman-Kac's formula \cite{fleming2006controlled}. This expectation (or path-integral) can then be approximated via Monte Carlo simulation, and since this defines the solution to the Hamilton-Jacobi-Bellman equation, it is a short leap from there to the optimal controller (or its approximation). This numerical algorithm for optimal control has received interest in, e.g., \cite{van2008graphical,todorov2009efficient,theodorou2010generalized,BroekUAI2010,morzfeld2014implicit} where a number of generalisations (and applications) have been investigated. To the best of our knowledge, no investigation of the stochastic stability of this approximation method has been considered.

The stability properties of this Monte Carlo based controller approximation method are also analysed. In particular, we relate this approximation method to the more general stability analysis provided in this work, and we show that this method stabilises the system (in a specific sense to be defined). This stability analysis justifies application of this control algorithm over an extended, possibly infinite, time interval. 

The remainder of this work is organised as follows: In Section 2 we outline the basic nonlinear stochastic RHC problem and some related notation. In Section 3 we consider the stability of the nonlinear stochastic RHC regime. In particular, we note the stabilisation properties of the optimal (ideal) controller and we analyze the stability properties of a number of controller approximation methods. In Section 4 we introduce the Monte Carlo based algorithm for controller approximation and we relate this algorithm and its stability properties to the results given in the previous section. In Section 5 we provide some concluding remarks and comment on a number of possible extensions.

\section{Nonlinear Stochastic Receding Horizon Optimal Control}

Let $(\Omega,\F,\PP)$ be a complete probability space equipped with the natural filtration $(\F_{t})_{t\geq0}$ generated by a fixed, standard, Wiener process $W_t(\omega):[0,\infty)\times\Omega\to\Rd$. We consider a nonlinear controlled process $X^{0,x_0,u}_t(\omega):[0,\infty)\times\Omega\to\Rn$
\begin{equation}
	dX^{0,x_{0},u}_{t} = f(X^{0,x_{0},u}_{t},u_{t})dt + g(X^{0,x_{0},u}_{t})dW_{t} \label{generalsystemcontrol}
\end{equation}
\noindent with $X^{0,x_{0},u}_{0} = x_{0} \in\Rn$. We assume $f:\Rn\times\Rm\to\Rn$ and $g:\Rn\to\Rnd$ to be continuous. The stochastic integrals in this paper are to be read in the Ito sense \cite{arnold1974stochastic}.  Moreover we assume that 
$$
	|f(x,u)-f(y,u)|+|g(x)-g(y)| \leq c_1|x-y|, \quad\forall(x,y,u)\in\Rn\times\Rn\times U 
$$
$$
	|f(x,u)-f(x,v)| \leq c_1|u-v|, \quad\forall(x,u,v)\in\Rn\times U\times U
$$
for some finite constant $c_1>0$. Let $t\geq s\geq0$, then the superscripts $X^{s,x,u}_t$ denote that the initial state at $s\geq0$ is $x$ and the control history is $(u_{t})_{t\geq s}$.

Fix a time interval $[t_0,t_1]$. Then a control $u_t(\omega):[t_0,t_1]\times\Omega\to U$ is said to be admissible if it is (progressively) Borel measurable and
$$
\E \left [ \int_{t_0}^{t_1} |u(X^{t_0,x,u}_s)|^q ds \right ] <\infty, \quad \forall x \in \R^n,~ q\geq 1
$$
We denote by $\U_{[t_0,t_1]}$ the class of admissible controls on $[t_0,t_1]$. These conditions are sufficient for the existence of a unique, continuous, (strong) solution to the stochastic process; e.g. see \cite{arnold1974stochastic, touzi2012optimal}.

Here we consider control and stabilisation to (a neighbourhood of) the origin; any other desired set point can be substituted via a simple change of coordinates. To this end, we fix $f(0,u(0)) = 0$, i.e. the origin is an equilibrium point for the nominal deterministic system. 

Let $T>0$ be fixed. We associate with (\ref{generalsystemcontrol}) the following receding horizon cost functional
\begin{equation}
	w(t,s,x, u) := \E\left[\phi(X^{t+s,x,u}_{T+t}) + \int_{t+s}^{T+t} \ell(X^{t+s,x,u}_{r},u_{r})dr\right],~~\quad \mathrm{for}~s\in[0,T] \nonumber
\end{equation}
\noindent where $\ell:\Rn\times\Rm\to[0,\infty)$ and $\phi:\Rn\to[0,\infty)$ are non-negative continuous functions that satisfy
$$
	c_2|x|^p \leq \phi(x) \leq c_3(1+|x|^p), \quad\forall x\in\Rn
$$
and
$$
	c_2(|x|^p + |u|^p) \leq \ell(x,u) \leq c_3(1+|x|^p+|u|^p), \quad\forall(x,u)\in\Rn\times U
$$
for some finite (independent) constants $c_2,c_3>0$ and $p\geq1$. Further, $\phi(0) = 0$ and $\ell(0,u(0))=0$.

We define a value functional as
\begin{equation}
	v(t,s,x) := \inf_{u_{r}\in\U_{[t+s,t+T]}}~w(t, s, x, u) ~=\inf_{u_{r}\in\U_{[t+s,t+T]}} \E\left[\phi(X^{t+s,x,u}_{T+t}) + \int_{t+s}^{T+t}\ell(X^{t+s,x,u}_{r},u_{r})dr\right] \label{value}
\end{equation}
and denote by $\overline{u}_r(x)$, if it exists, the optimal control, i.e. the admissible control process over the finite horizon $[t,T+t]$ that minimizes (\ref{value}). In (one-step) RHC it is necessary just to compute $\overline{u}_r(x)$ for $r=t$ at which point the cost functional (and thus the value functional) changes to capture the receding horizon.

The value functional is time-invariant with respect to the first argument, in the sense that
\begin{eqnarray}
	{v}(t,s,x) &=& \inf_{u_{r}\in\U_{[t+s,t+T]}} \E\left[\phi(X^{t+s,x,{u}}_{T+t}) + \int_{t+s}^{T+t}\ell(X^{t+s,x,{u}}_{r},{u}_r(x))dr\right] \nonumber \\
		&=&\, \inf_{u_{r}\in\U_{[s,T]}} \E\left[\phi(X^{s,x,{u}}_{T}) \,+\, \int_{s}^{T}\ell(X^{s,x,{u}}_{r},{u}_r(x))dr\right] ~=~ {v}(0,s,x) \nonumber
\end{eqnarray}

Note that when viewed over $s\in[0,T]$ the value function represents the so-called value-to-go over the fixed finite horizon $[t,T+t]$. Going forward, we often write $v(x)$ in place of $v(t,0,x)$ or $v(s,x)$ in place of $v(t,s,x)$ when dealing with the value-to-go function. 

With the modelling hypotheses adopted thus far, we have the following key lemma.

\begin{lemma} \label{lemmavbounds}
There exist a pair of positive constant $c_4, c_5$, depending only on $p, T, c_1,c_2,c_3$, such that
\begin{equation}
\label{lemma1}
	c_4|x|^p \leq v(x)\leq c_5(1+|x|^p),~ \quad \forall \, x \in \Rn 
\end{equation}\vspace{-0.1cm}
and thus $v(x)\rightarrow\infty$ with $|x|\rightarrow\infty$.
\end{lemma}

\begin{proof}
	The proof of the lemma is given in the appendix.
\end{proof}

Going forward we write $\partial_x v(x)$ and $\partial_{xx}v(x)$ for the gradient vector and Hessian matrix respectively. Furthermore,
$$
	\partial_{s}v(s,x) := \frac{\partial {v}(s,x)}{\partial s} = \frac{\partial}{\partial s} \,\inf_{u_{r}\in\U_{[s,T]}} \E\left[\phi(X^{s,x,{u}}_{T+t}) + \int_{s}^{T}\ell(X^{s,x,{u}}_{r},{u}_r(x))dr\right]
$$
where $\partial_{s}{v}(s,x)$ is defined on $s\in[0,T]$ for any $t\geq0$. We often write $\partial_{s}v(x)$ with $s\in[0, T]$ in place of $\partial_{s}v(s,x)$ for brevity.

Under certain conditions, at any time $s\in[0,T]$ the following Hamilton-Jacobi-Bellman (HJB) equation can be associated with the general value functional
\begin{equation}
	-\partial_{s}v(s,x) = \inf_{u\in\U}~ \left[\ell(x,u) + f(x,u)^\top \partial_x v(s,x) + \tfrac{1}{2}\tr\left[g(x)g(x)^\top\partial_{xx}v(s,x)\right]\right] \nonumber
\end{equation}
\noindent with a terminal boundary condition $v(T,x) = \phi(x)$. This association is in the sense that a suitably smooth solution to the HJB equation, if it exists, coincides with the value-to-go \cite{fleming2006controlled}. Note that in (one-step) RHC we are only interested in the solution of the HJB equation $ v(s,x)$ at time $s=0$ on the interval $s\in[0,T]$.

We note that assumptions introduced in this work are assumed to hold from the point at which they are introduced throughout the remainder of the work. The modelling hypotheses, e.g. on the functions $f$, $g$, $\phi$, $\ell$, etc. are assumed to hold throughout the remainder until the point they are refined (typically specialised) and from which point the refinement is supposed to hold.

\begin{assumption} \label{assump1}
	We assume that the modelling assumptions outlined to this point are augmented (where and how necessary) to ensure $v(s,x):[0,T]\times\Rn\rightarrow[0,\infty)$ is once continuously differentiable in $s\in[0,T]$ and twice continuously differentiable in $x$ for all $x\in\Rn\setminus\{0\}$ and that $v(s,x)$ is a solution to the corresponding associated HJB equation.
\end{assumption}

Sufficient conditions for this assumption to hold (in addition to the modelling hypotheses introduced thus far) are given in, e.g., \cite{fleming2006controlled}. Typically, these sufficient conditions take the form of further boundedness or regularity assumptions on the system and cost functions and/or their (partial) derivatives and are not overly restrictive\footnote{It is noteworthy that while a classical solution to the HJB equation arising in deterministic optimal control is not typical, it is well-known \cite{krylov1972control,fleming2006controlled,krylov2008controlled,touzi2012optimal} that the stochastic optimal control problem is quite generally `more regular'. Indeed, under an assumption of uniform parabolicity, i.e. uniform positive-definiteness of $g(x)g(x)^\top$, it generally follows that a classical (unique) solution to the HJB equation will exist in the stochastic setting (under mild regularity assumptions on the model/cost); see Chapter IV.4 in \cite{fleming2006controlled} or Krylov \cite{krylov2008controlled}. Separately, with an added Lipschitz assumption on the cost (compatible with the hypotheses herein), a classical solution to the HJB equation not only exists but is indeed Lipschitz \cite{yong1999stochastic,buckdahn2012regularity}. This Lipschitz setting is commonly assumed when studying the characteristics of stochastic optimal control; e.g. see \cite{yong1999stochastic}. Note, we do not generally require (or ask) for this Lipschitz property here. In any case, assumptions of this type concerning the existence of a classical solution are common in the analysis of both deterministic \cite{mayne1990receding,jadbabaie2001unconstrained} and stochastic optimal control \cite{yong1999stochastic}.}. We could also move from considering classical solutions of the HJB equation to generalised or viscosity solutions \cite{fleming2006controlled}.

Given the admissible optimal control process $\overline{u}_s(x)$ over the finite horizon $s\in[0,T]$ then the optimal value function over the finite horizon from any $t\geq0$ to $T+t$ is
$$
	v(x) = \E\left[\phi(X^{t,x,\overline{u}}_{T+t}) + \int_{t}^{T+t}\ell(X^{t,x,\overline{u}}_{s},\overline{u}_s(x))ds\right]
$$
and, given Assumption \ref{assump1}, the value-to-go satisfies the following HJB equation
\begin{eqnarray}
	-\partial_{s}v(s,x) = \ell(x,\overline{u}_s(x)) + f(x,\overline{u}_s(x))^\top \partial_x v(s,x)
		+ \tfrac{1}{2}\tr\left[g(x)g(x)^\top\partial_{xx}v(s,x)\right] \label{hjb}
\end{eqnarray}
\noindent on $s\in[0,T]$ with the terminal boundary condition $v(T,x) = \phi(x)$. We will use the following assumption.

\begin{assumption} \label{assump2}
	We assume $\partial_{s}v(x)|_{s=t}\geq0$ at any time $t\geq0$ where $s\in[t,T+t]$.
\end{assumption}

This assumption implies the optimal cost, when viewed at the start of a finite horizon, is increasing with decreasing horizon lengths\footnote{This assumption is common and discussed further in \cite{wei2014stability} and the references therein. It is proven to hold in \cite{magni1997stability} under quite typical modelling constraints (compatible with the modelling hypotheses presented here).}. One way to interpret this is that if the horizon length is reduced then the control action has less time to stabilize the system and thus the terminal cost is likely to be greater even though the running cost might be reduced. 

Finally, we highlight again that in optimal RHC, at any time $t\geq0$, the applied control is just ${u}^*_t(x) = \overline{u}_s(x)|_{s=0}$ and the remaining, finite horizon, controls $\overline{u}_s(x)$ over $0< s\leq T$ are discarded.

\section{Stability of Nonlinear Stochastic Receding Horizon Control and Robustness to Controller Errors}

It has been shown in \cite{wei2014stability} that nonlinear stochastic receding horizon control stabilises the system to the origin if the true optimal control is applied (and under comparable assumptions and hypotheses to those considered here). In this section, we generalise this result to the case in which approximations in computing the optimal control are naturally employed\footnote{In \cite{wei2014stability}, stability of the origin (under the exact controller) is studied under the classical requirement \cite{khasminskii2011stochastic} that $g(0)=0$, i.e. that the origin is an `exact' equilibrium for the diffusion, and thus the noise `goes to zero' at this point. Here, we consider the practical case in which approximations are made when computing the optimal control, and we study stability to some neighbourhood of the origin. Thus, we also do not require $g(0)=0$.}. 

Going forward we write $\B_\delta:=\{x\in\Rn:|x|\leq\delta\}$ for the $\delta>0$ ball around the origin and we use the shorthand $\{v<c\}$ to denote the level set $\{x\in\Rn:v(x) < c\}$ for $c\geq 0$. For any $\delta\geq0$ define $m_\delta:=\inf \{c\geq 0 ~|~ \B_\delta \subseteq \{v<c\}\}$. Informally stated, we study stability to $\{v<m_\delta\}$ where, from (\ref{lemmavbounds}), it follows that $\{v<m_\delta\}$ is contained in a ball around the origin with radius tending to $0$ as $\delta\rightarrow 0$.

Throughout this section we consider a number of general classes of controller approximation error; including deterministic and probabilistic errors and even controller sample and hold errors. We consider the controlled systems's stability robustness in the presence of these errors, and we show that the controller approximation errors do not accumulate (even over an infinite time frame) and the process converges exponentially fast to a small neighbourhood of the origin. These results capture the robustness of the underlying stability of the considered nonlinear stochastic receding horizon controller.

\subsection{Deterministic Control Errors}

We introduce the following control signal $\uu_t(x)\in\mathcal{U}$ with
$$
	|\uu_t({x})-u_t^*({x})|\leq \epsilon,~\quad\forall (x,t)\in\Rn\times[0,\infty)
$$
for some sufficiently small $\epsilon>0$. We denote by $\X^{0,x_{0},\uu}_{t}$ trajectories of (\ref{generalsystemcontrol}) driven by $\uu_t(x)$ with $\X_0=X_0=x_0$. If $\uu_t(x)\rightarrow  u_t^*(x)$ then $\X_t \rightarrow X^{0,x_{0},{u}^*}_{t}$ for all $t\geq0$; i.e. we recover the optimally controlled process in some suitable sense (to be made precise). The goal in this subsection is to prove that if $|\uu_t(x)-  u_t^*(x)|$ is small then $\X_t$ behaves similarly to $X_t$.

We will require the following assumption.

\begin{assumption}
\label{derivatives}
There exists a constant $c_6>0$ such that $|\partial_x v(x)| \leq c_6(1+|x|^p), \, \forall x\in\Rn$.
\end{assumption}

For generality, we have stated our requirement that $|\partial_x v(x)| \leq c_6(1+|x|^p)$ as an assumption. Nevertheless, results of this type, i.e. results concerning estimates/bounds of the derivative of the value function, are well studied\footnote{For example, it is proven in \cite{fleming2006controlled} that $|\partial_x v(x)| \leq c_6(1+|x|^p)$ holds under the given modelling hypotheses adopted in this work (on the cost/dynamics) with essentially the additional assumption $|\partial_{x} \ell(x,u)| \leq c(1+|x|^p + |u|^p)$, $c>0$. Thus, we already have the basic conditions on the dynamics/cost to ensure this assumption holds. See Chapter IV.8 in \cite{fleming2006controlled} and also Chapter 3 and Chapter 4 in \cite{krylov2008controlled}. Even stronger results have been proven \cite{fleming2006controlled,touzi2012optimal} implying this assumption holds trivially when one begins imposing Lipschitz conditions on the cost $\ell(x,u)$. } and conditions on $f$, $g$, $\ell$, and $\phi$ under which this assumption is guaranteed to hold are readily available \cite{fleming2006controlled,krylov2008controlled}.

\begin{theorem} \label{deterministicrobustnesstheorem}
Suppose Assumptions \ref{assump1}, \ref{assump2} and \ref{derivatives}, and the modelling hypotheses hold. Define $\beta:=c_5(1+\frac{1}{\delta^p})$ and $\lambda:=\frac{c_2}{\beta}$. Solutions of the SDE (\ref{generalsystemcontrol}) driven by the optimal control $u^*_t$ satisfy:
\begin{itemize}
\item if $x_0 \in \{v<m_\delta\}$ then, with probability one, $X^{0,x_0,u^*}_t$ will never exit $\{v<m_\delta\}$ and $\E|X^{0,x_0,u^*}_t|^p\leq \frac{m_\delta }{c_4}, ~\forall t\geq 0$;
\item if $x_0\notin \{v<m_\delta\}$, it holds
$$
\E|X^{0,x_0,u^*}_t|^p\leq  \frac{1}{c_4}\left(\beta e^{-\lambda t}|x_0|^p +m_\delta\right), \quad\quad \forall  t\geq 0,
$$
and, with probability one, $X^{0,x_0,u^*}_t$ hits $\{v < m_\delta\}$ in finite time.
\end{itemize}

These two points imply that almost all solutions to (\ref{generalsystemcontrol}) driven by the optimal control $u^*_t$ are exponentially stable to a ball around the origin and almost all trajectories remain within this ball.

Moreover there exists $\overline{\epsilon}>0$ such that if
$$
0<\epsilon <\overline{\epsilon}\quad \mbox{and} \quad |\uu_t({x})-u_t^*({x})|\leq \epsilon, \quad \forall x\in \Rn
$$
then solutions of the SDE (\ref{generalsystemcontrol}) driven by the control law $\uu_t$ satisfy the following:
\begin{itemize}
\item if $x_0 \in \{v<m_\delta\}$, then with probability one, $\X^{0,x_0,\uu}_t$ never exits $\{v<m_\delta\}$ and $\E|\X^{0,x_0,\uu}_t|^p\leq \frac{m_\delta }{c_4}, ~\forall t$;
\item if $x_0\notin\{v<m_\delta\}$, there exists a constant $\lambda\geq\theta_\epsilon>0$ such that
$$
\E|\X^{0,x_0,\uu}_t|^p\leq  \frac{1}{c_4}\left(\beta e^{-\theta_\epsilon t}|x_0|^p +m_\delta\right),  \quad\quad \forall t\geq 0,
$$
and the constant $\theta_\epsilon$ satisfies
$
\lim_{\epsilon\rightarrow 0} \theta_\epsilon=\lambda.
$
Further, with probability one, $\X^{0,x_0,\uu}_t$ hits $\{v < m_\delta\}$ in finite time.
\end{itemize}

These two points imply that almost all solutions to (\ref{generalsystemcontrol}) driven by the approximate control $\uu_t$ are exponentially stable to a ball around the origin and almost all trajectories remain within this ball. Such solutions converge exponentially fast under $\uu_t$ but slower than under $u^*_t$. In the limit $\epsilon\rightarrow0$ we recover the stability properties of the optimal controller.

\end{theorem}

\begin{proof}
Going forward, we use the shorthand $\X_t$ for the process (\ref{generalsystemcontrol}) driven by $\uu_t(\X_t)$. We prove only the second half of the theorem concerning the approximate controller $\uu_t$. Statements on the optimal control follow with $\epsilon=0$.

Under the hypotheses of the theorem it follows that
$
	c_4|x|^p\leq v(x) \leq c_5(1+|x|^p)
$
for some $c_2,c_4,c_5>0$. Let $\LL$ denote the infinitesimal generator \cite{khasminskii2011stochastic} of $X^{0,x_{0},{u}^*}_{t}$. Then
\begin{equation}
	\LL v :=  f(x,{u}^*_t(x))^\top \partial_x v + \tfrac{1}{2}\tr\left[g(x)g(x)^\top\partial_{xx}v\right] \nonumber
\end{equation}
is a function of $x\in\Rn$ derived by applying the infinitesimal generator to $v(x)$. From (\ref{hjb}), we have $\LL v = -\ell(x,{u}^*_t(x))-\partial_{s}v(s,x)|_{s=t}$ which by the modelling hypotheses and Assumption \ref{assump2} is strictly negative definite $\LL v< -c_2|x|^p$ for all $x\in\Rn\setminus\{0\}$ and $\LL v =0$ at $x=0$. 

Now, Ito's formula yields
$$
	dv(\X_t )~=~ \langle f(\X_t , \uu_t (\X_t )), \partial_x v(\X_t) \rangle dt +\tfrac{1}{2} \tr [ g(\X_t)g(\X_t)^\top \partial_{xx}v(\X_t)] dt + \langle g(\X_t)dW_t, \partial_x v(\X_t )\rangle
$$
Adding and subtracting $\langle f(\X_t , u_t^*(\X_t )), \partial_x  v(\X_t) \rangle$, and using the HJB equation (\ref{hjb}), we obtain
\begin{eqnarray}
	dv(\X_t) ~=~ \LL v(\X_t) dt + \langle  f(\X_t , \uu_t(\X_t)) -f(\X_t , u^*_t(\X_t )), \partial_x  v(\X_t) \rangle dt + \langle g(\X_t)dW_t, \partial_x v(\X_t ) \rangle \nonumber
\end{eqnarray}
Set
$$
	\widehat{\LL v}(x) :=\LL v(x) + \langle f(x , \uu_t(x ))-f(x ,u^*_t(x )), \partial_x v(x) \rangle
$$
where $\widehat{\LL v}(x)$ is the infinitesimal generator of $\X_t$ applied to $v(x)$, for any $x\in\Rn$.

Recall that the infinitesimal generator is a purely local construction \cite{khasminskii2011stochastic} which allows us to consider $\widehat{\LL v}(x)$ and $\LL v(x)< -c_2|x|^p$ at the same point in space-time. 

By using the Lipschitz condition on $f(x,u)$ and the Cauchy-Schwartz inequality we obtain
\begin{eqnarray}
	\label{LV} \widehat{\LL v}(x) & \leq& \LL v(x) + c_1|\partial_x v| |\uu_t(x)-u_t^{*}(x)| \nonumber \\
	&\leq& -c_2|x|^p + \epsilon c_1 |\partial_x v|  \nonumber \\
	& \leq& -c_2|x|^p+\epsilon c_1c_6(1+|x|^p) ~=~ (-c_2+\epsilon c_1c_6)|x|^p+\epsilon c_1 c_6 \nonumber
\end{eqnarray}

We define
$$
\alpha_\epsilon:= c_2-(\frac{1}{\delta^p}+1)\epsilon c_1c_6 \qquad \mbox{and} \qquad \theta_\epsilon:=\frac{\alpha_\epsilon}{\beta}
$$
There exists $\overline{\epsilon}>0$ small enough so $\alpha_\epsilon>0,~\forall \epsilon < \overline{\epsilon}$. Moreover we see that $\lim_{\epsilon \rightarrow 0} \alpha_\epsilon=c_2 \Rightarrow \lim_{\epsilon \rightarrow 0} \theta_\epsilon=\lambda$. Going forward we write $\alpha=\alpha_\epsilon$ for simplicity. It is easy to check that on $\{x\in\Rn:|x|>\delta\}$ we have
\begin{equation}
\label{Vestimate}
	v(x)\leq \beta|x|^p ~\quad\mathrm{and}\quad~ \widehat{\LL v}(x)\leq -\alpha|x|^p
\end{equation}
We define $V(t,x) := v(x) e^{\frac{\alpha t}{\beta}}$. Then, on the set $\{x\in\Rn:|x|>\delta\}$, it follows
\begin{equation}
\label{LVestimate}
	\widehat{\LL V}(t,x) ~=~ e^{\frac{\alpha t}{\beta}} \left( \frac{\alpha}{\beta}v(x) +\widehat{\LL v}(x) \right) ~\leq~ 0
\end{equation}
where $\widehat{\LL V}(t,x)$ is the infinitesimal generator of $\X_t$ applied to $V(t,x)$, for any $x\in\Rn$ at $t\geq0$.

Assume now $x_0\in \{v<m_\delta\}$. Given $t\geq0$, define the stopping times
$$
	\tau_1:=\inf\{ s\geq 0~|~ \X_s \notin \{v<m_\delta\}\} \wedge t \quad \mbox{and} \quad \tau_2:=\inf\{ s\geq \tau_1 ~|~ \X_s \in \{v<m_\delta\}\}\wedge t
$$
i.e. $\tau_1, \tau_2$ are, respectively, the first exit and re-entry time of the process $\X_t$ in $\{v<m_\delta\}$ before $t$. Going forward we write $V(\X_t)$ in place of $V(t,\X_t)$ for simplicity. By definition, for any $t_2\geq t_1$, we have
$$
		V(\X_{t_2})-V(\X_{t_1})=\int_{t_1}^{t_2} \widehat{\LL V}(\X_s)ds+\int_{t_1}^{t_2} e^{\frac{\alpha}{\beta}s} \partial_xv(\X_s)^\top g(\X_s)   dW_s.
$$
We know $\int_{t_1}^{t_2} e^{\frac{\alpha}{\beta}s} \partial_xv(\X_s)^\top g(\X_s)  dW_s$ is a martingale. Then, the optional sampling theorem \cite{doob1953stochastic} implies
$$
	\E\left[V(\X_{\tau_2})-V(\X_{\tau_1})\right]=\E\left[\int_{\tau_1}^{\tau_2} \widehat{\LL V}(\X_s)ds \right] \leq 0
$$
where the last inequality follows from (\ref{LVestimate}). We also note that, by definition, $V(\X_{\tau_2})\geq V(\X_{\tau_1})$. Therefore we have $V(\X_{\tau_2})= V(\X_{\tau_1})$ almost surely and consequently $\tau_1=\tau_2$ almost surely. Thus, if the process starts in the set $\{v<m_\delta\}$ it can never exit this set. It follows that
$$
	\E|\X_t|^p ~\leq~ \frac{1}{c_4}\E\left[v(\X_t)\right] ~\leq~ \frac{m_\delta }{c_4}
$$
and proof of the first point is complete.

Now assume $x_0\notin \{v<m_\delta\}$, fix $t$ and define the following stopping time
$$
	\tau:=\inf\{ s\geq 0~|~ \X_s \in \{v<m_\delta\}\} \wedge t
$$
Note that we have already considered the case $t\geq \tau \Rightarrow \X_t\in \{v<m_\delta\}$. Now write
$$
	V(\X_t)=v(x_0)+V(\X_\tau) - v(x_0) + V(\X_t) - V(\X_\tau)
$$
and take the expectation of both sides. Arguing as before $\E[V(\X_\tau) - v(x_0)]\leq 0$. Moreover
\begin{eqnarray}
	\E\left[V(\X_t) - V(\X_\tau)\right] ~=~ \E\left[e^{\frac{\alpha}{\beta}t}v(\X_t) -e^{\frac{\alpha}{\beta}\tau}v(\X_\tau)\right] ~\leq~ m_\delta\E\left[e^{\frac{\alpha}{\beta}t}-e^{\frac{\alpha}{\beta}\tau}\right] ~\leq~ m_\delta e^{\frac{\alpha}{\beta}t} \nonumber
\end{eqnarray}
Hence, using the two inequalities just shown, (\ref{Vestimate}) and (\ref{lemma1}) we have
\begin{eqnarray}
	\E|\X_t|^p &\leq& \frac{1}{c_4}\E\left[v(X_t)\right] ~=~ \frac{e^{-\frac{\alpha}{\beta}t}}{c_4}\E\left[V(\X_t)\right] ~\leq~ \frac{e^{-\frac{\alpha}{\beta}t}}{c_4} (v(x_0)+ m_\delta e^{\frac{\alpha}{\beta}t}) ~\leq~ \frac{\beta}{c_4} |x_0|^p e^{-\frac{\alpha}{\beta}t} + \frac{m_\delta}{c_4} \nonumber
\end{eqnarray}
and proof of so-called exponential p-stability is complete.

Now, we have already shown $ \E\left[V(\X_{\tau})-V(x_0)\right]=\E\left[\int_0^{\tau} \LL V(\X_s)ds \right] \leq 0$ which implies $\E\left[ v(\X_{\tau})e^{\gamma \tau} \right] \leq v(x_0)$. Further,
$$ 
\E\left[ v(\X_{\tau})e^{\gamma \tau }\right] ~=~\E \left[ v(\X_{\tau})e^{\gamma \tau} \mathbbm{1}_{\{\tau \neq t\}} + v(\X_{\tau})e^{\gamma \tau} \mathbbm{1}_{\{\tau =  t\}} \right] ~\geq~ \E\left[v(\X_{\tau})e^{\gamma \tau} \mathbbm{1}_{\{\tau =  t\}} \right] ~>~ m_\delta e^{\gamma t} \mathbb{P} (\tau = t)
$$
and thus $m_\delta e^{\gamma t} \mathbb{P} (\tau =t) < v(x_0)$ for all $t$. This implies $\mathbb{P} (\tau = t) \rightarrow 0$ as $t\rightarrow \infty$ which in turn implies that $ \mathbb{P} (\tau<\infty)=1$. From this and inequality (\ref{LVestimate}) it follows that $V(\tau, \X_{\tau})$ is a positive supermartingale. From Theorem 5.1 in \cite{khasminskii2011stochastic} it follows that $V(\tau, \X_{\tau})$ converges almost surely to a finite limit (dependent on $x_0$) as $t\rightarrow \infty$. Then, from (\ref{Vestimate}) we have
$$
|\X_{\tau}|^p \leq \frac{(\sup_t V(\tau, \X_{\tau}))}{c_4} e^{-\frac{\alpha}{\beta} \tau}
$$
with probability one. Letting $t\rightarrow \infty$ proves that almost all solutions converge exponentially fast toward $\{v < m_\delta\}$. Results of this type are known \cite{khasminskii2011stochastic}, i.e. where $p$-th moment exponential stability implies almost sure exponential stability.
\end{proof}

\subsection{Probabilistic Control Errors}

We now turn our attention to the case where the perturbed controller has a Gaussian distribution,
$$
	\uu_t(x) \sim \mathcal{N}(u_t^*(x),\Sigma(x))
$$
and the evolution of the nonlinear controlled process $X^{0,x_{0},u}_{t}(\omega):[0,\infty)\times\Omega\to\Rn$ follows
\begin{equation}
	dX^{0,x_{0},u}_{t} = f(X^{0,x_{0},u}_t)dt + h(X^{0,x_{0},u}_t)u_tdt + g(X^{0,x_{0},u}_t)dW_{t} \label{gaussianperturbedsystemcontrol}
\end{equation}
with the existing modelling hypotheses holding. Here $h:\Rn\to\Rnm$ is continuous with
$$
	|f(x)-f(y)|+|g(x)-g(y)| + |h(x)-h(y)| \leq c_1|x-y|, \quad\forall(x,y)\in\Rn\times\Rn
$$
for some finite constant $c_1>0$. We denote by $\X^{0,x_{0},\uu}_{t}$ trajectories of (\ref{gaussianperturbedsystemcontrol}) driven by $\uu_t(x)$ with $\X_0=X_0=x_0$.

In this subsection we seek a result analogous to Theorem \ref{deterministicrobustnesstheorem} under the proposed probabilistic controller error model. The goal is to show that if $\Sigma(x)\rightarrow0$ for all $x\in\Rn$ then $\uu_t(x)\rightarrow  u_t^*(x)$ and $\X_t \rightarrow X^{0,x_{0},{u}^*}_{t}$ for all $t\geq0$; i.e. we recover the optimally controlled process in some suitable sense. 

As before we need a further assumption on the derivatives of the value function.

\begin{assumption}
\label{derivatives2}
One of the two following condition holds:
\begin{itemize}
	\item there exists a constant $c_7>0$ such that $|\partial_{xx} v(x)| \leq c_7(1+|x|^{p-2}), \, \forall x\in\Rn$;
	\item there exists a constant $c_7>0$ such that $|\partial_{xx} v(x)| \leq c_7(1+|x|^p), \, \forall x\in\Rn$ and $h(x)$ is bounded.
\end{itemize}
\end{assumption}

As with Assumption \ref{derivatives}, we have stated our requirement that $|\partial_{xx} v(x)| \leq c_7(1+|x|^p)$ as an assumption (for the sake of generality). Yet similarly again, results of this type, i.e. results concerning estimates/bounds of the second derivative of the value function, are well studied in the literature\footnote{As with Assumption \ref{derivatives} it is proven in \cite{fleming2006controlled} that $|\partial_{xx} v(x)| \leq c_7(1+|x|^p)$ holds under the modelling hypotheses adopted in this work (on the cost/dynamics), with essentially the additional assumption that $|\partial_{x} \ell(x,u)| \leq c(1+|x|^p + |u|^p)$ and $|\partial_{xx} \ell(x,u)| \leq c(1+|x|^p + |u|^p)$, $c>0$. See Chapter IV.9 in \cite{fleming2006controlled} and also Chapter 3 and Chapter 4 in \cite{krylov2008controlled}. Again, stronger results have been proven \cite{fleming2006controlled,touzi2012optimal} implying this assumption holds trivially when one imposes Lipschitz conditions on the cost $\ell(x,u)$, which is common in similar analysis \cite{yong1999stochastic}.} \cite{fleming2006controlled,krylov2008controlled}.

The following is the main result of this subsection.

\begin{theorem} \label{probabilisticrobustnesstheorem}
Suppose Assumption \ref{assump1}, \ref{assump2} and \ref{derivatives2} and the relevant modelling hypotheses hold. Define $\beta:=c_5(1+\frac{1}{\delta^p})$ and $\lambda:=\frac{c_2}{\beta}$. Solutions of (\ref{gaussianperturbedsystemcontrol}) driven by the optimal control $u^*_t$ satisfy the relevant convergence results in Theorem \ref{deterministicrobustnesstheorem}.

 Moreover, there exists $\overline{\epsilon}>0$ such that if $~0<\epsilon <\overline{\epsilon}$ and the following holds $\forall x\in\Rn$
\begin{itemize}
\item $\uu_t(x) \sim \mathcal{N}(u_t^*(x),\Sigma(x))$;
\item $0<\Sigma(x) = \Sigma(x)^\top$~and~for~any~norm~$|\Sigma(x)|\leq \epsilon$
\end{itemize}
then solutions of the SDE (\ref{gaussianperturbedsystemcontrol}) driven by the approximated controller $\uu_t$ satisfy the following:
\begin{itemize}
\item if $x_0 \in \{v<m_\delta\}$, then with probability one, $\X^{0,x_0,\uu}_t$ never exits $\{v<m_\delta\}$ and $\E|\X^{0,x_0,\uu}_t|^p\leq \frac{m_\delta }{c_4},~\forall t$;
\item if $x_0\notin\{v<m_\delta\}$, there exists a constant $\lambda\geq\theta_\epsilon>0$ such that
$$
	\E|\X^{0,x_0,\uu}(t)|^p\leq  \frac{1}{c_4}\left(\beta e^{-\theta_\epsilon t}|x_0|^p +m_\delta\right),  \quad\quad \forall t\geq 0,
$$
and $\theta_\epsilon$ obeys $\lim_{\epsilon\rightarrow 0} \theta_\epsilon=\lambda$ (where $\lambda$ is the convergence rate of the optimal control; see Theorem \ref{deterministicrobustnesstheorem}). Further, with probability one, $\X^{0,x_0,\uu}_t$ hits $\{v < m_\delta\}$ in finite time.
\end{itemize}

Thus, almost all solutions to (\ref{gaussianperturbedsystemcontrol}) driven by the approximate control $\uu_t$ are exponentially stable to a ball around the origin and almost all trajectories remain within this ball. Such solutions converge exponentially fast under $\uu_t$ but slower than under $u^*_t$. As $\epsilon\rightarrow0$ we recover the stability properties of the optimal controller.
\end{theorem}

\begin{proof}
As before, denote by $\X_t$ the process (\ref{gaussianperturbedsystemcontrol}) driven by the approximated control $\uu_t(\X_t)$. We quickly find
$$
	d\X_t =  f(\X_t)dt + h(\X_t)u^*_tdt + h(\X_t)(\uu_t-u^*_t)dt + g(\X_t)dW_t
$$
Since $\Sigma(x)$ is (symmetric) positive-definite we have $\Sigma^{1/2}\Sigma^{1/2}=\Sigma(x)$ where $\Sigma^{1/2}$ exists and is unique. Then $\uu_t(x) \sim \mathcal{N}(u_t^*(x),\Sigma(x))$ implies $(\uu_t-u^*_t)dt=\Sigma^{1/2} dY_t$ where $Y_t$ is a standard Brownian motion \cite{arnold1974stochastic}. The two Brownian motions $Y_t$ and $W_t$ are realised on two different spaces: we have already fixed $\Omega$ and we denote by $\Omega'$ the space associated with the probabilistic controller approximation such that $[W_t^\top, Y_t^\top]^\top$ defines a fixed Brownian motion on $\Omega\times \Omega'$.

Let $v(x)=\E [\phi (X_T)+\int_0^T \ell(X_s,u^*_s)ds]$ where the process $X_t$ defining $v(x)$ is defined by (\ref{gaussianperturbedsystemcontrol}) driven with the optimal control $u^*_t(x)$. We consider
$$
	\widehat{\LL v} = \langle \,  f(x) + h(x) u^*_t, \partial _x v \, \rangle + \frac{1}{2} \tr [ g(x)g(x)^\top \partial_{xx} v] + \frac{1}{2} \tr [\Sigma(x) h(x)h(x)^\top \partial_{xx} v]
$$
where $\widehat{\LL v}(x)$ is the infinitesimal generator of $\X_t$ applied to $v(x)$, for any $x\in\Rn$ at $t\geq0$. Again, $\widehat{\LL v}(x)$ should be viewed as a function of $x\in\Rn$.

We know that
$$
	\langle \,  f(x) + h(x) u^*_t, \partial _x v \, \rangle + \frac{1}{2} \tr [ g(x)g(x)^\top \partial_{xx} v]  < -c_2|x|^p
$$
from the proof of Theorem \ref{deterministicrobustnesstheorem}, i.e. $\LL v(x)< -c_2|x|^p$. Owing to Assumption \ref{derivatives2} we have, for some positive constant $c$,
$$
	\frac{1}{2} \tr [\Sigma(x) h(x)h(x)^\top \partial_{xx} v] \leq \epsilon c (|x|^p+1)
$$
and therefore,
$$
	\widehat{\LL v}(x) ~\leq~ -c_2|x|^p + \epsilon c |x|^p+\epsilon c
$$
Define $\alpha_\epsilon:= c_2-(\frac{1}{\delta^p}+1)\epsilon c$ and the proof now follows exactly that of Theorem \ref{deterministicrobustnesstheorem} and we omit the repetition for brevity.
\end{proof}

\subsection{Mixed Type Errors and Sampled Control}

We now state a simple corollary that takes into account a mixed probabilistic and deterministic controller error.

\begin{corollary}
\label{mixederror}
Suppose we are working under (\ref{gaussianperturbedsystemcontrol}) and Assumptions \ref{assump1} to \ref{derivatives2} and the modelling hypotheses outlined thus far hold. Define $\beta:=c_5(1+\frac{1}{\delta^p})$ and $\lambda:=\frac{c_2}{\beta}$. There exist $\overline{\epsilon_1}, \overline{\epsilon_2} >0$ such that if $0<\epsilon_1<\overline{\epsilon_1}$ and $0<\epsilon_2<\overline{\epsilon_2}$ and 
\begin{itemize}
\item	$(\uu_t(x) - u_t^* (x)) \sim \mathcal{N}(\mu(x),\Sigma(x))$,
\item $|\mu(x)|\leq\epsilon_1$,
\item $0<\Sigma(x) = \Sigma(x)^\top$~and~for~any~norm~$|\Sigma(x)|\leq \epsilon_2$,
\end{itemize}
holds $\forall x \in \Rn$, then solutions of the SDE (\ref{gaussianperturbedsystemcontrol}) driven by the approximate control law $\uu_t$ satisfy:\begin{itemize}
\item if $x_0 \in \{v<m_\delta\}$, then with probability one, $\X^{0,x_0,\uu}_t$ never exits $\{v<m_\delta\}$ and $\E|\X^{0,x_0,\uu}_t|^p\leq \frac{m_\delta }{c_4},~\forall t\geq 0$;
\item if $x_0\notin\{v<m_\delta\}$, put $\epsilon=(\epsilon_1,\epsilon_2)$. There exists a constant $\lambda\geq\theta_\epsilon>0$ such that
$$
	\E|\X^{0,x_0,\uu}_t|^p\leq  \frac{1}{c_4}\left(\beta e^{-\theta_\epsilon t}|x_0|^p +m_\delta\right),   \quad\quad \forall ~ t\geq 0,
$$
and $\theta_\epsilon$ obeys $\lim_{\epsilon\rightarrow 0} \theta_\epsilon=\lambda$ (where $\lambda$ is the convergence rate of the optimal control; see Theorem \ref{deterministicrobustnesstheorem}). Also, with probability one, $\X^{0,x_0,\uu}_t$ hits $\{v < m_\delta\}$ in finite time, i.e. almost all solutions converge exponentially fast toward $\{v < m_\delta\}$.
\end{itemize}
\end{corollary}

\begin{proof}
If $\mu(x)\in\mathcal{B}_{\epsilon_1} = \{x\in\Rn:|x|\leq \epsilon_1\}$ then $(\uu_t(x) - u_t^* (x)) \sim \mathcal{N}(\mu(x),\Sigma(x))$ implies $(\uu_t-u^*_t)dt=\mu(x)+\Sigma(x)^{1/2} dY_t$ where $Y_t$ is a standard Brownian motion. It is then easily seen that the error is split in two parts, one part formed by the added Brownian motion and the other part formed by the deterministic error $\mu(x)$ with $|\mu(x)|\leq \epsilon_1$, $\forall x\in\Rn$. The proof of both Theorem \ref{deterministicrobustnesstheorem} and \ref{probabilisticrobustnesstheorem} apply readily in this case and we omit the details for brevity.
\end{proof}

In many practical scenarios it is impossible to compute the optimal control instantaneously and one must instead resort to a sample and hold approach to control whereby the control is computed at discrete-time increments and held constant in between such times. Stability results for such approaches have been considered, e.g., in deterministic settings \cite{mayne1990receding} and stochastic settings \cite{mahmood2012lyapunov}. We now provide a related stability result.

\begin{proposition}
\label{sampledcontrolprop}
Consider the more general controlled process (\ref{generalsystemcontrol}) and suppose Assumptions \ref{assump1}, \ref{assump2}, and the relevant modelling hypotheses hold. Suppose that under a given control law $u_t(x)$ the solution $X^{0,x_0,u}_t$ to the SDE (\ref{generalsystemcontrol}) with initial condition $x_0$ satisfies
\begin{equation}
\label{originalcontrol}
	\E|X^{0,x_0,u}(t)|^p \leq c e^{-\lambda t} |x_0|^p + m, \qquad \forall \mbox{ t} \geq 0
\end{equation}
for some positive constants $\lambda, c$ and $m$. Now fix a time step $\Delta>0$ and let the time interval $t\in[0,\infty)$ be discretised according to $t_0=0$, $t_1=\Delta$, $t_2=2\Delta$, $\ldots$, $t_k=k\Delta$. Consider the control law defined by $\uu_{t}(x_{t})=u_{t_k}(x_{t_k})$ for $t\in[t_k,t_{k+1})$, i.e. the control $\uu_t$ is held constant over small time intervals with a value given by the control $u_t$ at the beginning of each interval. Then there exists a constant step size $\overline{\Delta}>0$ and constants $M_1,M_2>0$ such that
\begin{equation}
	\E|\X^{0,x_0,\widehat{u}}(t)|^p \leq M_1 e^{-\frac{\lambda}{4} t} |x_0|^p + M_2, \qquad \forall \mbox{ t} \geq 0 \nonumber
\end{equation}
for all $0<\Delta \leq \overline{\Delta}$ and where we denote by $\X^{0,x_{0},\uu}_{t}$ trajectories of (\ref{generalsystemcontrol}) driven by $\uu_t(x)$ with $\X_0=X_0=x_0$.
\end{proposition}

\begin{proof}
Consider the two stochastic differential equations of the form (\ref{generalsystemcontrol}) but driven by the two different controls defined in the statement of the theorem
\begin{eqnarray}
	dX^{0,x_{0},u}_{t} &=& f(X^{0,x_{0},u}_{t},u_{t})dt + g(X^{0,x_{0},u}_{t},u_{t})dW_{t} \nonumber\\
	d\widehat{X}^{0,x_{0},\widehat{u}}_{t} &=& f(\widehat{X}^{0,x_{0},\widehat{u}}_{t},\widehat{u}_{t})dt + g(\widehat{X}^{0,x_{0},\widehat{u}}_{t},\widehat{u}_{t})dW_{t} \nonumber
\end{eqnarray}
Then, since both processes share a common initial point, it is straightforward to show that the two Euler-Maruyama time-discretisations of both processes are identical. That is, by induction on $k\in\mathbb{N}$ we have
\begin{eqnarray}
	\widehat{Z}^{0,x_{0},\widehat{u}}_{t_{k+1}} &=& \widehat{Z}^{0,x_{0},\widehat{u}}_{t_{k}} + f(\widehat{Z}^{0,x_{0},\widehat{u}},\widehat{u}_{t_k})\Delta + g(\widehat{Z}^{0,x_{0},\widehat{u}}_{t_k},\widehat{u}_{t_s})W_{\Delta} \nonumber \\
	&=& {Z}^{0,x_{0},{u}}_{t_{k}} + f({Z}^{0,x_{0},{u}},{u}_{t_k})\Delta + g({Z}^{0,x_{0},{u}}_{t_k},{u}_{t_k})W_{\Delta} ~=~ {Z}^{0,x_{0},{u}}_{t_{k+1}} \nonumber
\end{eqnarray}
where $W_{\Delta}\sim\mathcal{N}(0,\Delta\cdot \mathrm{I})$ and $Z_0=\widehat{Z}_0=x_0$.

There is a known result \cite{higham2003exponential} which states that $p$-th moment exponential stability of a stochastic differential equation implies $p$-th moment exponential stability of its Euler-Maruyama simulation and vice-versa (if the time-step $\Delta>0$ is sufficiently small). Thus, with minor modifications to the main result of \cite{higham2003exponential} it follows\footnote{The result in \cite{higham2003exponential} must be modified since here we consider exponential stability to a ball of the origin (not the origin itself as in \cite{higham2003exponential}). Thus, instead of the strong result of \cite{higham2003exponential}, we merely want exponential stability to the ball for a SDE to imply exponential stability to a (possibly different) ball for its Euler-Maruyama simulation (and vice-versa). The fact this relaxation is true follows easily (intuitively) given the strong result in \cite{higham2003exponential}. It is causally unsurprising. Details on the modifications required to relax \cite{higham2003exponential} as stated are available upon request (but needlessly distract the proof otherwise).} that if (\ref{originalcontrol}) holds, then for a sufficiently small step size $\Delta$, the Euler-Maruyama approximation ${Z}^{0,x_{0},{u}}_t$ of $X^{0,x_0,u}_t$ satisfies
$$
	\E|{Z}^{0,x_{0},{u}}_t|^p\leq c |x_0|^p e^{-\frac{1}{2} \lambda t}  + M, \qquad \forall \mbox{ t} \geq 0
$$
for some $M>0$. The same holds for $\widehat{Z}^{0,x_{0},\hat{u}}_t$ as this discrete-time process is identical to ${Z}^{0,x_{0},{u}}_t$. Again, with slight modification to the results in \cite{higham2003exponential} it follows that if $\Delta$ is small enough then
$$
	\E|\X^{0,x_0,\widehat{u}}_t|^p \leq M_1 e^{-\frac{\lambda}{4} t} |x_0|^p + M_2, \qquad \forall \mbox{ t} \geq 0
$$
for some positive constants $M_1,M_2$. This completes the proof.
\end{proof}

A straightforward consequence of Proposition \ref{sampledcontrolprop} is that the convergence results given thus far concerning the various controller approximation errors will continue to hold even if the control is computed only at discrete-time instants and held constant in the interval between such instants (provided that the time elapsed between each updates is small).

The next result brings everything together.

\begin{corollary}
\label{sampledcontrolcorollary}
Suppose the assumptions of either Theorem \ref{deterministicrobustnesstheorem}, Theorem \ref{probabilisticrobustnesstheorem} or Corollary \ref{mixederror} hold. Suppose also that $\uu_t$ is an approximately optimal control law satisfying the requirements of the respective result; e.g. $|\uu_t - u^*_t|< \epsilon < \overline{\epsilon}$ in Theorem \ref{deterministicrobustnesstheorem} etc. Fix $\delta>0$, we know that there exists a constant $\theta_\epsilon>0$, satisfying the statement of the respective result, such that
$$
	\E|\X^{0,x_0,\uu}(t)|^p\leq  \frac{1}{c_4}\left(\beta e^{-\theta_\epsilon t}|x_0|^p +m_\delta\right),   \quad\quad \forall ~ t\geq 0
$$
Now suppose that $\uu_t(x)$ is computed at discrete times $t_k$ with $t_0=0$, $t_1=\Delta$, $t_2=2\Delta$, $\ldots$, $t_k=k\Delta$ and held constant on the interval $t\in[t_k,t_{k+1})$ as described in Proposition \ref{sampledcontrolprop}. Then there exists a constant step size $\overline{\Delta}>0$ and constants $M_1,M_2>0$ such that
$$
	\E|\X^{0,x_0,\uu}(t)|^p\leq  M_1\beta e^{-\frac{\theta_\epsilon}{4} t}|x_0|^p +M_2m_\delta,  \quad\quad \forall ~ t\geq 0
$$
for all $0<\Delta \leq \overline{\Delta}$.
\end{corollary}

\section{Monte Carlo Methods for Approximately Optimal Stochastic Control}

In this section we outline an approximation method to compute the optimal nonlinear stochastic RHC. This method relies on simulating a stochastic process that is related to the original controlled system but that is independent of the control signal. The approximation method outlined in this section was first considered by Kappen \cite{kappen2005linear,kappen2005path} for finite-horizon optimal control and then subsequently studied, applied, and generalised in, e.g.,  \cite{van2008graphical,todorov2009efficient,theodorou2010generalized,BroekUAI2010,bierkens2014explicit,thijssen2014path}. 

Recall that we are considering the nonlinear controlled process $X^{0,x_{0},u}_{t}(\omega):[0,\infty)\times\Omega\to\Rn$ defined by
\begin{equation}
	dX^{0,x_{0},u}_{t} = f(X^{0,x_{0},u}_{t})dt + h(X^{0,x_{0},u}_{t})u_{t}dt + g(X^{0,x_{0},u}_{t})dW_{t}\label{montecarlosystemcontrol}
\end{equation}
with the existing modelling hypotheses holding. Here, $h(x)$ and $g(x)$ (which may be non-square) are assumed (with no real loss of generality) to have full rank. Note, $h(x)$ full rank implies the existence and uniqueness of a left-inverse, i.e. a function $h^{-1}(x) : \R^n\rightarrow \R^{m\times n}$ such that $h^{-1}h(x)={\mathrm{I}}, \forall x \in\R^n$. Associate with (\ref{montecarlosystemcontrol}) the following receding cost functional
\begin{equation}
	w(t, s, x, u) := \E\left[\phi(X^{t+s,x,u}_{t+T}) + \int_{t+s}^{t+T} \tfrac{1}{2}u_r^\top{R}u_r + \ell(X^{t+s,x,u}_{r})dr\right] \nonumber
\end{equation}
\noindent at any time $t\geq0$ with $s\in[0,T]$ and where the cost on the control input is now quadratic and $R\in\R^{m\times m}$ is a constant positive definite matrix. We define the value-to-go functional as
\begin{equation}
	v(t,s,x) := \inf_{u_{r}\in\U_{[t+s,t+T]}}~w(t,s, x, u) ~=\inf_{u_{r}\in\U_{[t+s,t+T]}} \E\left[\phi(X^{t+s,x,u}_{t+T}) + \int_{t+s}^{t+T} \tfrac{1}{2}u_r^\top{R}u_r + \ell(X^{t+s,x,u}_{r})dr\right] \label{montecarlovalue}
\end{equation}
where $\U_{[t+s,t+T]}$ is the set of admissible controls in the interval $[t+s,t+T]$. 

The HJB equation associated with the value functional (\ref{montecarlovalue}) is
\begin{equation}
	-\partial_{s}v(s,x) = \inf_{u\in\U}~ \left[\ell(x) + \tfrac{1}{2}u^\top R u + \left[f(x)+h(x)u\right]^\top \partial_x v(s,x) + \tfrac{1}{2}\tr\left[g(x)g(x)^\top\partial_{xx}v(s,x)\right]\right] \nonumber
\end{equation}
\noindent with a terminal boundary $v(T,x) = \phi(x)$. The optimal control on the interval defined by $s\in [0,T]$ is just $u^*_{t+s}(x) = -R^{-1}h(x)^\top \partial_x v(s,x)$ for all $x\in\Rn$. In (one-step) RHC we are only interested in the solution $v(s,x)$ at $s=0$. We have
$$
u^*_t(x) = -R^{-1}h(x)^\top \partial_x v(x), \quad \forall x\in\Rn
$$
Substituting the optimal control back into the HJB equation gives
\begin{equation}
	-\partial_{s}v(s,x) =  \ell(x) - \tfrac{1}{2} (\partial_x v(s,x))^\top h(x)R^{-1}h(x)^\top \partial_x v(s,x) + f(x)^\top \partial_x v(s,x) + \tfrac{1}{2}\tr\left[g(x)g(x)^\top\partial_{xx}v(s,x)\right]\nonumber
\end{equation}
which is a nonlinear partial differential equation. However, we note the following log-transform of $v(s,x)$
$$
	\psi(s,x) = \exp\left[\frac{-v(s,x)}{\gamma}\right]
$$
for all $x\in \Rn$, $s\in[0,T]$ and for some finite $\gamma>0$. This transform arises in a number of stochastic control scenarios \cite{fleming2006controlled}. We often write $\psi(x)$ in place of $\psi(0,x)$. We note the following required assumption.

\begin{assumption} \label{assump4}
	We assume that there exists $\gamma \in \mathbb{R}$ such that $\gamma h(x)R^{-1}h(x)^\top=g(x)g(x)^\top$.
\end{assumption}

This assumption\footnote{This assumption is satisfied in many applications of stochastic control; e.g. in machine learning and robotics \cite{kappen2005path,van2008graphical,todorov2009efficient,theodorou2010generalized,BroekUAI2010,bierkens2014explicit}. This assumption requires the dimension of the noise and control to be equal and for the noise and control to act on the same subspace. Then, the cost of control can be related to the noise variance as shown. The interpretation of this relationship is that along directions where the noise variance is small, the control is deemed more expensive while, conversely, in those directions in which the noise has larger variance the control is cheap \cite{kappen2005path}. Indeed, this may be desirable in practice since it forces control energy to be spent mostly in those directions in which the noise level may be problematic \cite{theodorou2010generalized}.} is standard in the path integral formulation of optimal control \cite{kappen2005path}, but it also appears more generally in the stochastic optimal control literature \cite{fleming2006controlled}. This assumption allows us \cite{fleming2006controlled,kappen2005path} to write
\begin{equation}
	-\partial_{s}\psi(s,x) =  -\frac{1}{\gamma}\ell(x)\psi(s,x) + f(x)^\top \partial_x \psi(s,x) + \tfrac{1}{2}\tr\left[g(x)g(x)^\top\partial_{xx}\psi(s,x)\right]\nonumber
\end{equation}
which is a linear partial differential equation on $[0,T]$ with terminal condition $\psi(T,x) = \exp[-\phi(x)/\gamma]$. It now follows by the Feynman-Kac formula that the solution to the above PDE at $(0,x)$ is given by
$$
	\psi(x) = \E\left[\exp\left(-\frac{1}{\gamma}\phi(Z^{t,x}_{T+t}) - \frac{1}{\gamma}\int_{t}^{T+t} \ell(Z^{t,x}_s)ds \right) \right] 
$$
where now $Z^{t,x}_s(\omega):[t,T+t]\times\Omega\to\Rn$ is a nonlinear (uncontrolled) process satisfying
\begin{equation}
	dZ^{t,x}_{s} = f(Z^{t,x}_s)ds + g(Z^{t,x}_s)dW_s 	\label{montecarlosystemuncontrolled}
\end{equation}
with initial condition $Z_t^{t,x}=x$. Note that
$$
	u^*_t(x) = -R^{-1}h(x)^\top \partial_x v(x) = \gamma R^{-1}h(x)^\top \partial_x \log\psi(x)
$$
Now, given the solution for $\psi(x)$ derived via the Feynman-Kac formula, it is informally straightforward to devise a Monte Carlo approximation for the control; e.g. one can first simulate sample paths of (\ref{montecarlosystemuncontrolled}), then form a Monte Carlo approximation of the integral for $\psi(x)$, and approximate the spatial derivative of $\psi(x)$ via differencing. Going forward we explore a more formal Monte Carlo approximation circumventing the need for crude numerical (spatial) differentiation. Firstly, we need the following result. 

\begin{proposition} \label{optimalcontrolMonteCarlo}
	Suppose Assumptions \ref{assump1} and \ref{assump4}, and the modelling hypotheses hold. Then
\begin{equation}
\label{control formula1}
u^*_t(x) ~=~ -R^{-1}h(x)^\top \partial_x v(x) ~=~ \lim_{r\rightarrow 0} \frac{1}{r} \frac{\E\left[e^{-\frac{1}{\lambda}\left(\phi (Z^{t,x}_{t+T}) +\int_{t}^{t+T} \ell (Z^{t,x}_s)ds \right)} \int_0^r h^{-1}(Z^{t,x}_s)g(Z^{t,x}_s)dW_s   \right]}{\E\left[e^{-\frac{1}{\lambda}\left(\phi(Z^{t,x}_{t+T}) +\int_{t}^{T+t} \ell (Z^{t,x}_s)ds \right)}\right]}
\end{equation}
where the expectations are integrals over paths defined by the SDE (\ref{montecarlosystemuncontrolled}) with initial condition $Z^{t,x}_t=x$.
\end{proposition}
\begin{proof}
	This result appears in \cite{thijssen2014path} with $h(x)=g(x)$ and it is straightforward to generalise.
\end{proof}

The controller form in Proposition \ref{optimalcontrolMonteCarlo} (and variations of such) is often referred to as the path integral formulation of optimal control \cite{kappen2005path}. At this stage, it may appear as though the reformulated optimal controller has been significantly complicated. However, the optimal control as given in Proposition \ref{optimalcontrolMonteCarlo} is well suited to Monte Carlo approximation.

The Monte Carlo approach to RHC is given by Algorithm \ref{algorithm1}. Note also that we consider two time-discretizations defined by $\Delta_1>0$ and $\Delta_2>0$ respectively. The first, $\Delta_1$, captures the sample and hold application in which the control is computed at discrete-time steps and held constant over those intervals; i.e. we approximate the optimal control $u^*_t(x)$ by $\widehat{u}_{t}(\x_{t}) = \widehat{u}_{t_k}(\x_{t_k})$ over $t\in[t_k,t_{k+1})=[k\Delta_1,(k+1)\Delta_1)$. We denote by $\X^{0,x_{0},\uu}_{t}$ trajectories of (\ref{montecarlosystemcontrol}) with $\X_0=X_0=x_0$ driven by $\uu_{t}(x_{t})$. The second time-discretization, $\Delta_2$, is found solely within Algorithm \ref{algorithm1} and defines the time-step employed during the numerical simulation of (\ref{montecarlosystemuncontrolled}) used to actually compute $\widehat{u}_{t_k}(\x_{t_k})$ at each $t_k$.

\begin{table}[!ht]
\noindent\makebox[\linewidth]{\rule{\textwidth}{1pt}}
\caption{{\footnotesize Optimal Control Approximation via Monte Carlo Simulation}} \label{algorithm1}
\noindent\makebox[\linewidth]{\rule{\textwidth}{1pt}}

{\footnotesize
\begin{description}
  \item[Given at time t=0:] \hfill
  	\begin{enumerate}
    	\item Model hypotheses: $f(\cdot)$, $g(\cdot)$, $h(\cdot)$, $\phi(\cdot)$, $\ell(\cdot)$, $R$, $T$, and $\gamma$.
		\item Initial starting point: $x_0\in\mathbb{R}^n$.
		\item Discretization of time $t\in[0,\infty)$ via $t_0=0$, $t_1=\Delta_1$, $t_2=2\Delta_1$, $\ldots$, $t_k=k\Delta_1$.
        \item Discretization of the interval $[0,T]$ with step-size $\Delta_2$ such that $T/\Delta_2=K\in\mathbb{N}$.
	\item Parameter approximating the limit $r>0$ such that $r/\Delta_2=R\in \mathbb{N}$.
	\end{enumerate}

  \item[Available at time $t_k$] \hfill
  	\begin{enumerate}
		\item Current state: $\x_{t_k}\in\mathbb{R}^n$.
	\end{enumerate}

  \item[At time $t_k$ do:] \hfill
	\begin{enumerate}
    	\item Simulate $N$ times the following discrete-time approximation of (\ref{montecarlosystemuncontrolled})
		$$
			Z^{0,\x_{t_k}}_{t_{s+1}} = Z^{0,\x_{t_k}}_{t_{s}} + f(Z^{0,\x_{t_k}}_{s})\Delta_2 + g(Z^{0,\x_{t_k}}_{t_s})W^{\Delta_2}_{t_s}
		$$
		over $t_s\in\{0,\Delta_2, \ldots, s\Delta_2,\ldots\,K\Delta_2\}$ where $W^{\Delta_2}_{t_s}\sim\mathcal{N}(0,\Delta_2\cdot\mathrm{I})$. Simulation can be parallelised.
	\item Let
		$$
			z^{0,\x_{t_k}}_{0:K}(i) := \{z^{0,\x_{t_k}}_{0}(i)=\x_{t_k},\, z^{0,\x_{t_k}}_2(i),\, \ldots,\, z^{0,\x_{t_k}}_K(i)  \}
		$$
		be the ordered set of sample points along the simulated discretised trajectory on the $i^{th}$ simulation run.
		\item For each sampled trajectory $i\in\{1,\ldots, N\}$ compute
		$$
			\widehat{W}(i)=\sum_{j=1}^R h^{-1}(z^{0,\x_{t_k}}_{j-1}(i))g(z^{0,\x_{t_k}}_{j-1}(i))\left(w^{\Delta_2}_j(i) -w^{\Delta_2}_{j-1}(i)\right)
        		$$		
		where $w^{\Delta_2}_j(i)$ are the sample points of $W^{\Delta_2}_{t_j}$ used previously to generate the trajectory $z^{0,\x_{t_k}}_{0:K}(i)$.
		\item For each sampled trajectory $i\in\{1,\ldots, N\}$ compute
		\begin{eqnarray}
			\eta(i)=\phi(z_K^{0,\x_{t_k}}(i))\, + \sum_{j=0}^{K-1} \ell(z_j^{0,\x_{t_k}}(i))\Delta_2 \quad\nonumber
		\end{eqnarray}
		\item Compute
		$$
			\widehat{u}_{t_k}(\x_{t_k}) = \frac{1}{\sum_{i=1}^N \exp[-\tfrac{1}{\gamma} \eta(i)]} \sum_{i=1}^N \exp[-\tfrac{1}{\gamma}
        \eta(i)]\,\frac{\widehat{W}(i)}{r} 
		$$
		which gives a (naive) Monte Carlo approximation of the optimal control. Let $\widehat{u}_{t}(\x_{t}) = \widehat{u}_{t_k}(\x_{t_k})$ over $t\in[t_k,t_{k+1}) = [k\Delta_1,(k+1)\Delta_1)$.
	\end{enumerate}
\end{description}
}
\noindent\makebox[\linewidth]{\rule{\textwidth}{1pt}}
\end{table}

This algorithm is easily implementable. The numerical approximation of the stochastic differential equation (\ref{montecarlosystemuncontrolled}) is known as the Euler-Maruyama method and is the simplest numerical scheme for approximating stochastic differential equations. This numerical approximation may be generalised \cite{kloeden1992numerical} although care must be taken to ensure that sufficient gains warrant the sharp increase in complexity that accompanies higher-order numerical approximation schemes.

The error in computing the approximate control signal at the discrete time sites is a mix of the error introduced due to the Monte Carlo sampling (known as the statistical error) and the error introduced due to the approximation of the limit and the time-discretisation (known as the discretisation error); see also \cite{giles2008multilevel,bertoli2015error}. At those specific discretised time sites we note the following result.

\begin{proposition}\label{centrallimitprop}
Suppose Assumptions \ref{assump1} and \ref{assump4} and the modelling hypotheses employed to this point hold. Suppose also that the system and value functionals are sufficiently regular. Given $x\in\Rn$, suppose Algorithm \ref{algorithm1} is used to compute ${\uu}_t(x)$. Then there exists a positive constant $\mu_{\Delta_2}$, a function $\mu(x)$ satisfying $|\mu(x)|\leq \mu_{\Delta_2}$ and a matrix $\Sigma(x)$ such that
$$
	\sqrt{N}\left (\widehat{u}_t(x)-u^*_t(x)-\mu(x)\right) ~{\rightarrow} ~~\mathcal{N} (0,\Sigma(x))
$$
where convergence is `in distribution' with the number, $N$, of Monte Carlo runs; see Algorithm \ref{algorithm1}. Also, $\lim_{\Delta_2\rightarrow 0} \mu_{\Delta_2}=0$.
\end{proposition}

\begin{proof}
	Let $u^*_t$ denote the optimal control defined by (\ref{control formula1}). For a fixed $r>0$ approximate the limit defining
$$
u^*_{t,r}=\frac{1}{r} \frac{\E\left[e^{-\frac{1}{\lambda}\left(\phi (Z^{t,x}_{t+T}) +\int_{t}^{t+T} \ell (Z^{t,x}_\tau)d\tau \right)} \int_0^r h^{-1}(Z^{t,x}_\tau)g(Z^{t,x}_\tau)dW_\tau   \right]}{\E\left[e^{-\frac{1}{\lambda}\left(\phi(Z^{t,x}_{t+T}) +\int_{t}^{T+t} \ell (Z^{t,x}_\tau)d\tau \right)}\right]}
$$
For $r$ small enough we have $|u^*_t(x) - u^*_{t,r}(x)| < \epsilon$ with $\epsilon$ to be chosen later. Let $\widetilde{u}^*_t$ be the approximation to $u^*_{t,r}$ found purely as a result of the discretized path approximation (associated with step-size $\Delta_2$). Then, given the convergence results for the Euler-Maruyama method \cite{kloeden1992numerical, bertoli2015error}, it follows that for $\Delta_2$ small enough, there exists a constant $c$ such that $|u^*_{t,r}(x) - \widetilde{u}^*_t(x)| \leq c \Delta_2 $ for all $r>0$. Using the triangular inequality, 
\begin{equation}
\label{weakconvergence}
  |u^*_t(x) - \widetilde{u}^*_t(x)| \leq c\Delta_2 + \epsilon=:\mu_{\Delta_2}
\end{equation}
Choosing $\epsilon\simeq\Delta_2$ yields  $\lim_{\Delta_2\rightarrow 0} \mu_{\Delta_2}=0$. We note that $\widehat{u}_t(x)$ is a Monte Carlo approximation of $\widetilde{u}^*_t(x)$. We denote by $W_{0:K}$ a realised sequence of the discretized Brownian motion associated with the Euler-Maruyama discretization of (\ref{montecarlosystemuncontrolled}) and by $\widetilde{Z}^{0,x}_{0:K}(W_{0:K})$ the discrete path associated with it. Call $\mathbb{P}$ the natural measure on the path space $\{W_{0:K}\}$. Define
$$
G(W_{0:K}):=\exp\left[-\frac{1}{\gamma} \left( \phi(\widetilde{Z}^{0,x}_K(W_{0:K})) + \sum_{j=1}^{K-1}\ell(\widetilde{Z}^{0,x}_j(W_{0:K}))\Delta_2\right) \right]
$$
and consider the path measure $\mathbb{Q}$ obtained by the relation
$$
d\mathbb{Q} = \frac{G(W_{0:K})}{\mathbb{E}_\mathbb{P} [G(W_{0:K})]} d\mathbb{P}
$$
Define the function $F(W_{0:K})=\sum_{j=1}^R\,h^{-1}g(\widetilde{Z}^{0,x}_j(W_{0:K}))\, (W_j - W_{j-1})$, i.e. the sum of the first $R$ Brownian increments. We have that
$$
r\widehat{u}_t(x)=\mathbb{E}_\mathbb{Q} [F(W_{0:K})]
$$
	
When simulating paths in Algorithm \ref{algorithm1}, we simulate from the measure $d\widetilde{\mathbb{Q}}:=G(W_{0:K})d\mathbb{P}$ and use self-normalized importance sampling to compute $r\widehat{u}_t(x)$. We know \cite{cappe2005inference}, that self-normalized importance sampling is asymptotically unbiased and moreover a central limit theorem holds if
$$
\int[1+F^2]\left(\frac{d\mathbb{Q}}{d\widetilde{\mathbb{Q}}}\right)^2d\widetilde{\mathbb{Q}} <\infty
$$
Here, we have
$$
\int[1+F^2]\left(\frac{d\mathbb{Q}}{d\widetilde{\mathbb{Q}}}\right)^2d\widetilde{\mathbb{Q}}  = \frac{1}{\E_\mathbb{P}[G(W_{0:K})]^2}\left(\int (1+ F^2)d\widetilde{\mathbb{Q}}\right)
$$
Moreover $\int F^2d\widetilde{\mathbb{Q}}=\E_\mathbb{P}[F^2G]<\infty$ thanks to the fact that $G$ is bounded. Therefore we have
$$
r\sqrt{N}\left (\widehat{u}_t(x)-\widetilde{u}^*_t(x)\right) ~{\rightarrow}~~ \mathcal{N} (0,\Sigma(x))
$$
where $\Sigma(x)= \int\left(\frac{d\mathbb{Q}}{d\widetilde{\mathbb{Q}}}\right)^2[F-\mathbb{Q}(F)]^2d\widetilde{\mathbb{Q}}$. Convergence is in the sense of distribution with $N$. Now divide by $r$, add and subtract $u^*_t(x)$, call $\mu(x)=u^*_t(x)- \widetilde{u}^*_t(x)$ and use (\ref{weakconvergence}) to prove the convergence result. 
\end{proof}

The asymptotic bias in the preceding error result can be reduced by decreasing $\Delta_2$ or via a reduction in the horizon length $T$. The variance can be reduced by increasing $N$ or through some variation of naive sampling such as improved importance sampling or additionally via particle methods and resampling schemes \cite{cappe2005inference,kappen2005path,theodorou2012relative,morzfeld2014implicit} etc. The role of the parameter $r$ with respect to the variance and the bias in the error approximation can be important; see \cite{bertoli2015error} for a first study of this issue. Note also that 
$$
\Sigma(x)=\frac{\var_\mu(F)}{\E_\mathbb{P}[G]}
$$
and therefore the variance is intimately connected to $\E_\mathbb{P}[G]$, i.e. the interplay between the cost and dynamics of the uncontrolled SDE. Such performance questions may be explored in future work; see also \cite{thijssen2014path,bertoli2015error}.

Going forward with the analysis we use the following assumption.

\begin{assumption}
\label{normaldistribution}
	Suppose that, for $N$ big enough, $(\widehat{u}_t(x)-u^*_t(x)) \sim \mathcal{N} (\mu(x),\frac{1}{N}\Sigma(x))$ and $\Sigma(x)$ is bounded.
\end{assumption}

This assumption is just an invocation of the central limit type of result in Proposition \ref{centrallimitprop} (which states that with $N$ increasing, the distribution of the random part of the control approximation can be assumed Gaussian)\footnote{The point of this assumption is to impose normality on the distribution of the error $\widehat{u}_t-u^*_t$. Regardless of the distribution, it is true that the variance of the error decreases proportionally with increasing $N$ (at the rate $1/N$) and that the bias decreases continuously with $\Delta_2$. In practice, with $N$ large enough, any error in applying this assumption is small and can be quantified via bounds of the Berry-Esseen type \cite{cappe2005inference}.}.

We can now state the main stability result of this section.

\begin{theorem}
Suppose Assumptions \ref{assump1}, \ref{assump2}, \ref{derivatives}, \ref{derivatives2}, \ref{assump4}, \ref{normaldistribution} and the modelling hypotheses outlined to this point hold. Define $\beta:=c_5(1+\frac{1}{\delta^p})$ and $\lambda:=\frac{c_2}{\beta}$. Given $x\in\Rn$, suppose Algorithm \ref{algorithm1} is used to compute ${\uu}_t(x)$. With $\Delta_1,\Delta_2>0$ small enough and $N$ large enough, there exits $\lambda\geq\theta>0$ and a pair of positive constants $M_1,M_2>0$ such that
$$
\E|\X^{0,x_0,\uu}(t)|^p\leq  M_1\beta e^{-\frac{\theta}{4} t}|x_0|^p +M_2m_\delta,   \quad\quad \forall ~ t\geq 0
$$
and $\lim_{\Delta_1,\Delta_2\rightarrow 0, N\rightarrow\infty} \theta=\lambda$ (where $\lambda$ is the rate of convergence of the optimal control; see Theorem \ref{deterministicrobustnesstheorem}).
\end{theorem}

\begin{proof}
Since the error $\widehat{u}_{t_k}-u^*_{t_k}$ is of the mixed type, we call on Corollary \ref{mixederror}. From Proposition \ref{centrallimitprop} it follows that there exists $\Delta_2$ small enough so that the deterministic part of the controller approximation error is small. Similarly, from Proposition \ref{centrallimitprop} it follows that there exists $N$ large enough so that the variance $\Sigma(x)$ is small. Assumption \ref{normaldistribution} imposes normality on the error distribution. Corollary \ref{mixederror} applies immediately. Picking $\Delta_1$ small enough to invoke Corollary \ref{sampledcontrolcorollary} completes the proof.
\end{proof}

\section{Concluding Remarks}

 In this work we explored the stability and the convergence properties of nonlinear stochastic RHC when the optimal controller is computed only approximately. We considered a number of general classes of controller approximation error including deterministic and probabilistic errors and even controller sample and hold errors. In each case, it is shown that the controller approximation errors do not accumulate (even over an infinite time frame) and the process converges exponentially fast to a small neighbourhood of the origin. We also overviewed an approximation method for computing the optimal RHC for nonlinear stochastic continuous-time systems. This method is based on Monte Carlo integration approximation and originates in the work of Kappen \cite{kappen2005linear,kappen2005path}.

 While we study the stability of various RHC approximations, we did not consider any measure of performance. For example, it would be of interest to analyze (path-wise) the running cost error that arises due to the approximation of the optimal controller. Inverse optimality and optimality gaps as studied in \cite{magni1997stability} would also be of interest here.

The incorporation of state constraints in RHC is common \cite{mayne2000constrained}. We note that it may be natural in some cases to incorporate state constraints in the Monte Carlo based approximation algorithm detailed herein. For example, state constraints may be enforced by simply restricting the evolution of the sampled trajectories (e.g. via dictating that certain regions of the state space hold zero probability).

Efficient sampling and Monte Carlo simulation \cite{kappen2005path,morzfeld2014implicit} that reduces the variance and thus the error in the Monte Carlo based controller approximation is of interest. Other computational aspects of this approximation are of interest, particularly as they apply to high-dimensional implementation.

Finally, we mention that extensions which account for partial-information feedback may be important, particularly in the stochastic framework where true state feedback is overly restrictive. In this setting, the coupling of stochastic RHC, and particularly the Monte Carlo approximation algorithm, with sequential Monte Carlo estimation/filtering (e.g. particle filtering \cite{cappe2005inference}) would be a natural topic for further study. Extensions to more general dynamical model settings may also be considered; e.g. systems with time-varying delays, high-order stochastic systems, etc.

\section*{Appendix: Proof of Lemma \ref{lemmavbounds}}

We start with the lower bound. Recall that $v(x)=\E[ \phi (X_T)+\int_0^T \ell(X_s, u^*_s)ds]$ where we use the shorthand $X_s$ for the solution of the system (\ref{generalsystemcontrol}) driven by the optimal control with initial condition $x$. We have
$$
	\E[\phi (X_T)] ~\geq~ c_2 \E|X_T|^p ~\geq~ c_2| \E[X_T]|^p
$$
using the modelling hypotheses first and Jensen's inequality second. Moreover, we have
\begin{eqnarray}
	\E \left[\int_0^T\ell(X_s, u^*_s)ds \right] & \geq & c_2 \E\left[\int_0^T (|X_s|^p + |u^*_s|^p)ds \right] \qquad\mathrm{Hypothesis:}~c_2(|x|^p + |u|^p) \leq \ell(x,u)\nonumber \\
	& \geq & c_2 2^{1-p} \E\left[\int_0^T (|X_s| + |u^*_s|)^pds \right] \quad~\mathrm{Hypothesis:~equivalence~of~norms~in~} \mathbb{R}^2 \nonumber \\
	& \geq & \frac{c_2 2^{1-p} }{c_1} \E\left[\int_0^T |f(X_s,u^*_s)|^pds\right] \qquad\mathrm{Hypothesis:~see~below} \nonumber \\
	&\geq & \frac{ c_2 2^{1-p} }{T^{p-1}c_1} \E \left[\Big|\int_0^T f(X_x,u^*_s)ds\Big |^p \right] \qquad \mathrm{Jensen's~inequality~on~inner~integral}\nonumber \\
	&\geq& \frac{ c_2 2^{1-p} }{T^{p-1}c_1}\Big| \E\left[\int_0^T f(X_x,u^*_s)ds\right]\Big|^p \qquad \mathrm{Jensen's~inequality~on~outer~integral}\nonumber\\
	&=& \frac{ c_2 2^{1-p} }{T^{p-1}c_1} \Big| \E [X_T-x]\Big|^p \qquad \mathrm{Taking~the~expectation~of~the~SDE}\nonumber\\
	&=&  \frac{ c_2 2^{1-p} }{T^{p-1}c_1} \Big| x-\E [X_T]\Big|^p  \qquad\mathrm{Initial~condition~is~deterministic} \nonumber
\end{eqnarray}
\noindent where we have used the modelling hypotheses on $\ell(x,u)$, together with the fact that $|f(x,u)|\leq |f(x,u) - f(0,u)| + |f(0,u) - f(0,0) | \leq c_1(|x| + |u|)$ and where we used Jensen's inequality twice. Putting together the bounds on $\E[\phi (X_T)]$ and $\E[\int_0^T\ell(X_s, u_s)ds]$ gives
\begin{eqnarray}
	v(x) &\geq& c_2 |\E[X_T]|^p + \frac{ c_2 2^{1-p} }{T^{p-1}c_1} | x-\E [X_T]|^p ~\geq~ c \Big ( |\E[X_T]|+|x-\E[X_T]| \Big)^p ~\geq~ c \Big ( |\E[X_T]+x-\E[X_T]| \Big)^p  ~=~ c|x|^p \nonumber
\end{eqnarray}
for some constant $c>0$, where we used the fact the all norms are equivalent on a finite dimensional vector space followed by the triangle inequality. This completes the proof for the lower bound.

We now turn to the upper bound and recall that, given an arbitrary admissible control law, the cost functional $w(t, x, u)$ associated with (\ref{generalsystemcontrol}) is
$$
	w(0, x, u) := \E\left[\phi(X_{T}) + \int_0^T \ell(X_{s},u_{s})ds\right]
$$
We immediately have
$$
v(x) := \inf_{u_{s}\in\U_{[0,T]}}~w(0, x, u) \leq w(0, x, 0)
$$
where $w(0, x, 0)$ denotes the cost found after applying a constant zero control $u_t\equiv 0$. Going forward, write $X_t :=X^{0,x,0}_t$ for the solution of (\ref{generalsystemcontrol}) with $X_0=X^{0,x,0}_0=x$ and a constant zero control $u_t\equiv 0$. Then
$$
	dX_t = {f}(X_t,0)dt + {g}(X_t,0)dW_t
$$
Note that if $u_t=0$ is not admissible we may substitute some other (sub-optimal) constant control signal. Then, for all $p\geq1$ with $t\geq 0$, its known \cite{touzi2012optimal} that given the existence of solutions to (\ref{generalsystemcontrol}) it holds that
$$
	\E|X_t |^p \leq c(1+|x|^p)e^{ct}
$$
\noindent for some finite $c>0$. This, together with the assumptions on the cost, gives
\begin{eqnarray}
w(0, x, 0) &=& \E\left[\phi(X_{T}) + \int_0^T \ell(X_{s},0)ds\right] \nonumber\\
        & \leq& \E\left[ c_3(1+|X_T|^p)\right] + \E\left[ \int_0^T c_3(1+|X_s|^p + 0)ds\right] \nonumber\\
        &\leq& c_3\left(1+T + |x|^pe^{cT} + \int_0^T|x|^pe^{cs}ds\right) ~\leq~ c_3\left(1+T + |x|^pe^{cT} +|x|^p\frac{e^{cT}-1}{c}\right) \nonumber
\end{eqnarray}
which, after gathering constants, completes the proof concerning the upper-bound.

Bringing everything together, it follows that there exists a pair of positive constants $c_4, c_5$, depending only on $p, T$, $c_1$, $c_2,c_3$, such that $c_4|x|^p \leq v(x)\leq c_5(1+|x|^p),~ \forall \, x \in \Rn$ and $v(x)\rightarrow\infty$ with $|x|\rightarrow\infty$. \qed

\end{document}